\documentclass[11pt]{article}

\usepackage[mathscr]{eucal}
\usepackage{amsmath}
\usepackage{amssymb}
\usepackage{amsfonts}
\usepackage{amscd}
\usepackage{amsthm}
\usepackage{latexsym}
\usepackage{graphicx}

\def\be{\begin{equation}}
\def\ee{\end{equation}}
\def\bse{\begin{subequations}}
\def\ese{\end{subequations}}

\newtheorem{thm}{Theorem}

\newtheorem{lem}{Lemma}[section]
\newtheorem{prop}{Proposition}[section]

\def\bse{\begin{subequations}}
\def\ese{\end{subequations}}

\usepackage{geometry}

\def\XXint#1#2#3{{
\setbox0=\hbox{$#1{#2#3}{\int}$}
\vcenter{\hbox{$#2#3$}}\kern-.5\wd0}}
\title{Residual vanishing for blowup solutions to 2D Smoluchowski-Poisson equation}

\author{Takashi Suzuki} 

\date{\today}

\begin{document}
\let\cleardoublepage\clearpage%svinei kenes selides

\maketitle

%\begin{center} 
%{\it dedicated to Professor Toshitaka Nagai}
%\end{center} 

\begin{abstract}
We study Smoluchowski-Poisson equation in two space dimensions provided with Dirichlet boundary condition for the Poisson part.  For this equation several profiles of blowup solution have been noticed.  Here we show the residual vanishing. 
\end{abstract}

\section{Introduction}\label{sec:1}

We study parabolic-elliptic system  proposed in statistical physics to describe the motion of mean field of many self-gravitating Brownian particles \cite{sc02}.  It is composed of the Smoluchowski part 
\begin{equation}
u_t=\Delta u-\nabla \cdot u\nabla v \quad \mbox{in $\Omega\times(0,T)$} 
 \label{s}
\end{equation} 
with null-flux boundary condition 
\begin{equation}
\frac{\partial u}{\partial \nu}-u\frac{\partial v}{\partial\nu}=0 \quad \mbox{on $\partial\Omega \times (0,T)$} 
 \label{eqn:s3} 
\end{equation}
and the Poisson part in the form of 
\begin{equation}
-\Delta v=u, \quad \left. v\right\vert_{\partial\Omega}=0, 
 \label{eqn:s4} 
\end{equation}
where $\Omega\subset {\bf R}^2$ is a bounded domain with smooth boundary $\partial\Omega$ and $\nu$ is the outer unit normal vector.  Initial condition is given as 
\begin{equation} 
\left. u\right\vert_{t=0}=u_{0}(x)\geq 0 \quad \mbox{in $\Omega$}, 
 \label{eqn:s2}
\end{equation} 
where $u_0=u_0(x)$ is a smooth function.  

System (\ref{s})-(\ref{eqn:s2}) is subject to thermodynamical laws, {\it total mass conservation} and {\it free energy decreasing}, 
\begin{eqnarray} 
& & \frac{d}{dt}\int_\Omega u=\int_\Omega \nabla\cdot (\nabla u-u\nabla v)=\int_{\partial\Omega}\frac{\partial u}{\partial \nu}-u\frac{\partial v}{\partial\nu} \ ds=0 
 \label{eqn:tmc} \\ 
& & \frac{d}{dt}{\cal F}(u)=-\int_\Omega u\vert \nabla (\log u-v)\vert^2\leq 0  
 \label{eqn:fed}
\end{eqnarray} 
where $ds$ denotes the surface element and 
\begin{equation} 
{\cal F}(u)=\int_\Omega u(\log u-1) \ dx -\frac{1}{2}\langle (-\Delta)^{-1}u,u\rangle 
 \label{eqn:free-e}
\end{equation} 
with $v=(-\Delta)^{-1}u$ standing for (\ref{eqn:s4}).  

A related model arises in the context of chemotaxis in theoretical biology \cite{jl92, nag95}, where the Poisson part is provided with the Neumann boundary condition such as 
\begin{equation}
-\Delta v=u-\frac{1}{\vert\Omega\vert}\int_\Omega u, \quad \left. \frac{\partial v}{\partial \nu}\right\vert_{\partial\Omega}=0, \quad \int_\Omega v=0. 
 \label{eqn:sd} 
\end{equation}
Concerning (\ref{s})-(\ref{eqn:s3}), (\ref{eqn:s2}), and (\ref{eqn:sd}), there is a threshold of $\Vert u_0\Vert_1=\lambda$ for the blowup of the solution.  More precisely, if $\lambda<4\pi$ the solution exists global-in-time \cite{bil98, gz98, nsy97}.  If a local mass greater than $4\pi$ is concentrated on a boundary point, on the contrary, there arises blowup in finite time \cite{nag01, ss01b}. Underlying blowup mechanisms were suspected from the study of stationary solutions \cite{cp81}. This attempt was followed by \cite{hv96, ss00}, using radially symmetric and general stationary solutions, respectively.  Up to now several properties have been known \cite{ss01a, sen07, ns08}.  See our previous work \cite{s14} and the references therein.  System (\ref{s})-(\ref{eqn:s2}), provided with Dirichlet condition for the Poisson part, is taken by \cite{s13}. It excludes boundary blowup points. Here we continue the study \cite{s14} on interior blowup points.  

Fundamental features of system (\ref{s})-(\ref{eqn:s2}) are the following. First, local-in-time unique existence of the classical solution is standard, given smooth initial value $u_0=u_0(x)\geq 0$.  Henceforth, $T\in (0, +\infty]$ denotes its maximal existence time. If $u_0\not\equiv 0$, which we always assume, the strong maximum principle and the Hopf lemma guarantee $u(\cdot,t)>0$ on $\overline{\Omega}$ for $t>0$.  Maximal existence time $T$ of non-stationary solution $u=u(\cdot,t)$, on the other hand, is estimated from below by $\Vert u_0\Vert_\infty$.  Hence $T<+\infty$ implies 
\[ \lim_{t\uparrow T}\Vert u(\cdot,t)\Vert_\infty=+\infty \] 
and the blowup set 
\begin{equation} 
{\cal S}=\{ x_0\in\overline{\Omega} \mid \mbox{$\exists x_k\rightarrow x_0$, $\exists t_k\uparrow T$ such that $u(x_k,t_k)\rightarrow+\infty$}\}  
 \label{eqn:blowupset1}
\end{equation} 
is non-empty.  Since boundary blowup points are excluded \cite{s13}, if $T<+\infty$ in (\ref{s})-(\ref{eqn:s2}) we have 
\begin{equation} 
u(x,t)dx\rightharpoonup \sum_{x_0\in {\cal S}}m(x_0)\delta_{x_0}(dx)+f(x)dx \quad \mbox{in ${\cal M}(\overline{\Omega})=C(\overline{\Omega})'$} 
 \label{eqn:bfq}
\end{equation} 
as $t\uparrow T$.  Here, the blowup set satisfies ${\cal S}\subset \Omega$ with $\sharp {\cal S}<+\infty$, and it holds that $0\leq f=f(x)\in L^1(\Omega)\cap C(\overline{\Omega}\setminus {\cal S})$.  

The blowup mechanism at each inner blowup point, described by \cite{s14}, is more complicated than the ones suspected before. Let $x_0\in{\cal S}$ and $R(t)=(T-t)^{1/2}$. As is shown in \cite{suzuki05, suzuki11} it holds that 
\begin{equation} 
\lim_{b\uparrow+\infty}\limsup_{t\uparrow T}\left\vert \Vert u(\cdot,t)\Vert_{L^1(B(x_0, bR(t))}-m(x_0)\right\vert =0. 
 \label{eqn:ext8}
\end{equation} 
Henceforth, $C_i$, $i=1,2,\cdots, 15$, denote positive constants. 

\begin{thm}[\cite{s14}]
Any $t_k\uparrow T$ admits a sub-sequence, denoted by the same symbol, and $m\in {\bf N}\cup \{0\}$, provided with the following property. Thus, given $0<\varepsilon\ll 1$ and $R\gg 1$, there is $\tilde s\geq 1$ such that, for $t_k'\uparrow T$ defined by $R(t_k')=\tilde sR(t_k)$ we have $0<r^j_k\leq C_1R(t_k')$ and $x^j_k\in B(x_0, C_1R(t_k'))$, $1\leq j\leq m$.  Then, for $B_k^j=B(x_k^j, b_j)$ and $\displaystyle{ G_k=\bigcup_{j=1}^mB_k^j}$, it holds that 
\begin{eqnarray} 
& & B_k^i\cap B_k^j=\emptyset, \quad i\neq j, \ k \gg 1 \nonumber\\ 
& & \limsup_{k\rightarrow\infty}\left\vert \Vert u(\cdot,t_k')\Vert_{L^1(B_k^j)}-8\pi\right\vert <\varepsilon, \quad 1\leq j\leq m \nonumber\\ 
& & \lim_{b\uparrow+\infty}\limsup_{k\rightarrow \infty}R(t_k')^2\Vert u(\cdot,t_k')\Vert_{L^\infty(B(x_0, bR(t_k'))\setminus G_k)}\leq C_2R^{-2}. 
 \label{eqn:1.209}
\end{eqnarray} 
 \label{thm:1.11}
\end{thm}

In this paper we show that if (\ref{eqn:1.209}) is extended continuously then it holds that $m(x_0)\in 8\pi {\bf N}$.  Namely, here we assume that any $\varepsilon>0$ admits $m_j(t)\in {\bf N}\cup \{ 0\}$, $x_j(t)\in B(x_0, C_3R(t))$, and $0<r_j(t)\leq C_3R(t)$ for $0<T-t\ll 1$ and $1\leq j\leq m(t)\in {\bf N}\cup \{0\}$, such that, for $B_j(t)=B(x_j(t), r_j(t))$ and $\displaystyle{ G(t)=\bigcup_{j=1}^{m(t)}B_j(t)}$ it holds that 
\begin{eqnarray} 
& & B_i(t)\cap B_j(t)=\emptyset, \quad i\neq j \nonumber\\ 
& & \limsup_{t\uparrow T}\max_{1\leq j\leq m(t)}\left\vert \Vert u(\cdot,t)\Vert_{L^1(B_j(t))}-8\pi\right\vert <\varepsilon \nonumber\\ 
& & \lim_{b\uparrow+\infty}\limsup_{t\uparrow T}R(t)^2\Vert u(\cdot,t)\Vert_{L^\infty(B(x_0, bR(t))\setminus G(t))}\leq C_4.  
 \label{eqn:1.209ih}
\end{eqnarray}  
Then our result arises as follows. 

\begin{thm}
If (\ref{eqn:1.209ih}) holds in (\ref{s})-(\ref{eqn:s2}) with $T<+\infty$, then there is $\tilde t_k\uparrow T$ such that 
\begin{equation} 
\lim_{b\uparrow +\infty}\lim_{k\rightarrow\infty}\Vert u(\cdot,\tilde t_k)\Vert_{L^1(B(x_0, bR(\tilde t_k))\setminus G(\tilde t_k))}<\varepsilon. 
 \label{eqn:2.189ih}
\end{equation} 
Hence we obtain $m(x_0)=8\pi m$, $m\in {\bf N}$, in (\ref{eqn:bfq}), and in particular, $m(t)=m$ in (\ref{eqn:1.209ih}). 
 \label{thm:2.14ih}
\end{thm} 

We call (\ref{eqn:2.189ih}) the {\it residual vanishing}. Although $m(x_0)\in 8\pi {\bf N}$ follows from (\ref{eqn:ext8}), (\ref{eqn:1.209ih}), and (\ref{eqn:2.189ih}), we shall show $m(x_0)\in 8\pi {\bf N}$ first, and then (\ref{eqn:2.189ih}). In future we shall discuss the problems to derive (\ref{eqn:1.209ih}) from (\ref{eqn:1.209}) and also to refine (\ref{eqn:2.189ih}) to 
\[ \displaystyle{ \lim_{b\uparrow+\infty}\limsup_{t\uparrow T}\Vert u(\cdot,t)\Vert_{L^1(B(x_0, bR(t))\setminus G(t))}<\varepsilon}. \] 

This paper is composed of three sections. Taking preliminaries in the next section, we prove Theorem \ref{thm:2.14ih} in the final section. 

\section{Prelimiaries}

Weak solution is a fundamental tool in later arguments. This notion was introduced first for the pre-scaled Smoluchowski-Poisson equation \cite{ss02a}. Let $G=G(x,x')$ be the Green's function to (\ref{eqn:s4}).  Then we say that $0\leq \mu=\mu(dx,t)\in C_\ast([0,T], {\cal M}(\overline{\Omega}))$ is a weak solution to (\ref{s})-(\ref{eqn:s4}) if there is 
\[ 0\leq \nu=\nu(\cdot,t)\in L^\infty_\ast(0,T; {\cal E}') \] 
called multiplicate operator satisfying the following properties, where ${\cal E}$ is the closure of the linear space 
\[ {\cal E}_0=\{ \psi+\rho_\varphi \mid \psi\in C(\overline{\Omega}\times\overline{\Omega}), \ \varphi\in {\cal X} \}, \quad {\cal X}=\{ \varphi\in C^2(\overline{\Omega}) \mid \left. \frac{\partial\varphi}{\partial \nu}\right\vert_{\partial\Omega}=0\} \] 
in $L^\infty(\Omega\times\Omega)$ and $\rho_\varphi(x,x')=\nabla\varphi(x)\cdot \nabla_xG(x,x')+\nabla\varphi(x')\cdot\nabla_{x'}G(x,x')$: 
\begin{enumerate} 
\item For $\varphi\in {\cal X}$ the mapping $t\in [0,T]\mapsto \langle \varphi, \mu(dx,t)\rangle$ is absolutely continuous and there holds 
\begin{equation} 
\frac{d}{dt}\langle \varphi, \mu(dx,t)\rangle=\langle \Delta \varphi, \mu(dx,t) \rangle+\frac{1}{2}\langle \rho_\varphi, \nu(\cdot,t) \rangle_{{\cal E}, {\cal E}'} \quad \mbox{a.e. $t$}. 
 \label{eqn:mul1}
\end{equation} 
\item We have 
\begin{equation} 
\left. \nu(\cdot,t)\right\vert_{C(\overline{\Omega}\times\overline{\Omega})}=\mu(dx,t)\otimes \mu(dx',t) \quad \mbox{a.e. $t$}.  
 \label{eqn:mul2}
\end{equation} 
\end{enumerate} 
Here we confirm that the property $\nu\geq 0$ of $\nu\in {\cal E}'$ means 
\[ \left\vert \langle f, \nu\rangle_{{\cal E}, {\cal E}'}\right\vert\leq \langle g, \nu \rangle \] 
for any $f, g\in {\cal E}$ satisfying $\vert f\vert\leq g$ a.e. in $\Omega\times\Omega$.  

We note the following properties.  First, the total mass conservation of this weak solution  
\[ \mu(\overline{\Omega},t)=\mu(\overline{\Omega}, 0), \quad t\in [0,T]  \] 
is obvious.  Next, this weak solution cannot be a measure-valued solution constructed in \cite{ds09, lsv12, ssv13}.  In fact, any collision of collapses are not admitted here, and more precisely, we have the following property.  
\begin{lem}[\cite{ss02a}] 
If the initial meause $\mu_0(dx)\in {\cal M}(\overline{\Omega})$ admits $x_0\in \Omega$ such that 
\[ \mu_0(\{ x_0\})>8\pi, \quad \lim_{R\downarrow 0}\frac{1}{R^2}\left\langle \vert x-x_0\vert^2\chi_{B(x_0,R)}, \mu_0(dx)\right\rangle =0 \] 
then there is no weak solution to (\ref{s})-(\ref{eqn:s4}) even local-in-time.   \label{thm:ss02a}
\end{lem} 
It is, however, provided with the following property, derived from the fact that ${\cal E}$ is separable.  

\begin{lem}[\cite{ss02a}]
Let $\{ \mu_k(dx,t)\}\subset C_\ast([0,T], {\cal M}(\overline{\Omega}))$ be a sequence of weak solutions to (\ref{s})-(\ref{eqn:s4}). Let the associated multiplicate operator of $\mu_k(dx,t)$ be $\nu_k(\cdot,t)\in L^\infty_\ast(0,T; {\cal E}')$, and assume 
\begin{equation} 
\mu_{k}(\overline{\Omega}, 0)+\sup_{t\in [0,T]}\Vert \nu_k(\cdot,t)\Vert_{{\cal E}'} \leq C_{5}, \quad k=1,2,\cdots. 
 \label{eqn:w-bound}
\end{equation} 
Then we have a subsequence denoted by the same symbol, $\mu(dx,t)\in C_\ast([0,T], {\cal M}(\overline{\Omega}))$, and $\nu(\cdot,t)\in L^\infty_\ast(0, T; {\cal E}')$ such that 
\begin{eqnarray*} 
& & \mu_k(dx,t)\rightharpoonup \mu(dx,t) \quad \mbox{in $C_\ast([0,T], {\cal M}(\overline{\Omega}))$} \\ 
& & \nu_k(\cdot,t)\rightharpoonup \nu(\cdot,t) \quad \qquad \mbox{in $L^\infty_\ast(0,T; {\cal E}')$}  
\end{eqnarray*} 
up to a sub-sequence, and this $\mu(dx,t)$ is a weak solution to (\ref{s})-(\ref{eqn:s4}) with the multiplicate operator $\nu(\cdot,t)$ satisfying 
\[ \mu(\overline{\Omega},0)+\Vert \nu(\cdot,t)\Vert_{{\cal E}'}\leq C_{5}.  \] 
 \label{pro:1}
\end{lem} 

We agree with the following notations. First, if $\mu(dx,t)$ has a density as 
\[ \mu(dx,t)=u(x,t)dx, \quad 0\leq u(\cdot,t)\in C([0,T], L^1(\Omega)), \] 
then the multiplicate operator is always taken as 
\[ \nu(\cdot,t)=u(x,t)u(x',t) \ dxdx', \] 
recalling ${\cal E}\subset L^\infty(\Omega\times\Omega)$.  Under this agreement, condition (\ref{eqn:w-bound}) is reduced to 
\begin{equation} 
\mu_{k}(\overline{\Omega},0)\leq C_{6}, \quad k=1,2,\cdots 
 \label{eqn:sk1} 
\end{equation}
if each $\mu_k(dx,t)$ takes density in $[0,T)$ such as 
\[ \mu_k(dx,t)=u_k(x,t)dx, \quad 0\leq u_k=u_k(\cdot,t)\in C([0,T), L^1(\Omega)). \] 
In fact, since we take 
\[ \nu_k(\cdot,t)=u_k(x,t)u_k(x',t)dxdx' \] 
in this case, inequality (\ref{eqn:sk1}) means 
\[ \Vert u_k(\cdot,t)\Vert_1=\Vert u_{k}(\cdot,0)\Vert_1=\mu_{k}(\overline{\Omega},0)\equiv \lambda_k \leq C_{6}. \] 
Therefore, (\ref{eqn:w-bound}) follows with 
\[ \Vert \nu(\cdot,t)\Vert_{{\cal E}'}=\lambda_k^2\leq C_{6}^2. \] 
Using this property, we shall derive a hierarchy of weak solutions in later arguments.  

We can define also the regularity of the weak solution.  First, given $\mu=\mu(\cdot, t)\in {\cal M}(\overline{\Omega})$, we have a unique $v=v(\cdot,t)\in W^{1,q}(\Omega)$, $1\leq q<2$, such that 
\[ -\Delta v=\mu, \quad \left. v\right\vert_{\partial\Omega}=0. \] 
Let $I\subset (0,T)$ be an open interval and $\omega\subset \Omega$ an open set. If the weak solution $\mu(dx,t)$ has a density $u=u(\cdot,t)\in L^p(\omega)$ in $\omega\subset \Omega$, $1<p<\infty$, for $t\in I$, the above $v=v(\cdot,t)$ is in $W^{2,p}_{loc}(\omega)$ from the elliptic regularity. By Sobolev's and Morrey's imbedding theorems, this regularity implies $(u\nabla v)(\cdot,t)\in L^1_{loc}(\omega)$.  Then we assign 
\[ \frac{d}{dt}\langle \varphi, \mu(dx,t)\rangle=\langle \Delta \varphi(dx,t)\rangle+\langle \nabla \varphi\cdot\nabla v, \mu(dx,t)\rangle, \quad \mbox{a.e. $t\in I$} \] 
for any $\varphi\in C_0^2(\omega)$.  In such a case we say that $\mu(dx,t)$ is regular in $\omega\times I$. In Lemma \ref{pro:1}, if $\mu_k(dx,t)$ is regular with the density $u_k(x,t)$ in $\omega\times(0,T)$ satisfying 
\[ \sup_{t\in [0,T]}\Vert u_k(\cdot,t)\Vert_{L^p(\omega)}\leq C_{7} \] 
for $p>1$, then the generated $\mu(dx,t)$ is also regular in $\omega\times (0,T)$.  Conversely, we have the following properties by the $\varepsilon$ regularity \cite{ss02b, s14}.  First, if the weak solution $\mu(dx,t)\in C_\ast([0,T], {\cal M}(\overline{\Omega}))$ is generated by a sequence of classical solutions $\{u_k(x,t)\}$ then its singular part $\mu_s(dx,t)$ is composed of a finite sum of delta functions.  Furthermore, if $\mu(dx,t_0)$, $0<t_0<T$, is regular in the sense of measure in an open set $\hat \omega\subset \Omega$, then it is regular in the above sense.  More precisely, $\mu(dx,t)$ takes a smooth density function $f=f(x,t)$ in $\omega\times (t_0-\delta, t_0+\delta)$, where $\omega$ is an open set satisfying $\overline{\omega}\subset \hat \omega$ and $0<\delta \ll 1$. 

The weak solution $a(dx,t)\in C_\ast(-\infty,+\infty; {\cal M}({\bf R}^2))$ to 
\begin{equation} 
a_t=\Delta a-\nabla\cdot a\nabla\Gamma\ast a \quad \mbox{in ${\bf R}^2\times (-\infty, +\infty)$} 
 \label{eqn:entire}
\end{equation}
is defined similarly, where $\displaystyle{ \Gamma(x)=\frac{1}{2\pi}\log\frac{1}{\vert x\vert}}$ is the fundamental solution to $-\Delta$, ${\cal M}({\bf R}^2)=C_0({\bf R}^2)'$ with 
\[ C_0({\bf R}^2)=\{ f\in C({\bf R}^2\bigcup\{\infty\}) \mid f(\infty)=0 \}, \] 
and ${\bf R}^2\bigcup \{\infty\}$ denotes one-point compactification of ${\bf R}^2$. Then the following property is shown, see \cite{s14}. 

\begin{prop}[Liouville property]
Let $0\leq a=a(dx,t)\in C_\ast((-\infty,+\infty), {\cal M}({\bf R}^2))$ be a weak solution to (\ref{eqn:entire}) with uniformly bounded multiplicate operator. Then we have either $a({\bf R}^2,t)=8\pi$ or $a({\bf R}^2, t)=0$, exclusively in $t\in {\bf R}$. 
 \label{lem:1.2} 
\end{prop} 

We can also define the weak solution $\zeta(dy,s)\in C_\ast(-\infty,+\infty; {\cal M}({\bf R}^2))$ to 
\begin{equation} 
\zeta_s=\Delta \zeta-\nabla\cdot \zeta\nabla(\Gamma\ast \zeta+\vert y\vert^2/4) \quad \mbox{in ${\bf R}^2\times (-\infty, +\infty)$},   
 \label{eqn:wsl1}
\end{equation} 
which arises as the weak scaling limit of $u=u(x,t)$.  Thus, given $x_0\in{\cal S}\subset \Omega$, we take the backward self-similar transformation 
\[ z(y,s)=(T-t)u(x,t), \quad y=(x-x_0)/(T-t)^{1/2}, \quad s=-\log(T-t). \] 
Let $t_k\uparrow+\infty$ and put 
\[ s_k=-\log (T-t_k) \ \uparrow+\infty. \] 
Then, passing to a sub-sequence denoted by the same symbol, we have 
\begin{equation} 
z(y, s+s_k)dy\rightharpoonup \zeta(dy,s) \quad \mbox{in $C_\ast(-\infty, +\infty; {\cal M}({\bf R}^2)$}), 
 \label{eqn:zeta}
\end{equation} 
where $\zeta(dy,s)$ is weak solution to (\ref{eqn:wsl1}) provided with a uniformly bounded multiplicate operator.  In (\ref{eqn:zeta}), the important property called {\it parabolic envelope} arises as 
\begin{equation} 
\zeta({\bf R}^2, s)=m(x_0)>0, \quad \langle \vert y\vert^2, \zeta(dy,s)\rangle \leq C_{8}
 \label{eqn:pe}
\end{equation} 
valid to $s\in (-\infty, +\infty)$ with $m(x_0)>0$ defined by (\ref{eqn:bfq}), see \cite{suzuki05, s13}. 

Here we take the {\it scaling back} of $\zeta(dy,s)$, defined by the transformation 
\begin{equation} 
A(dy', s')=e^s\zeta(dy,s), \quad y'=e^{-s/2}y, \quad s'=-e^{-s}. 
 \label{eqn:c9}
\end{equation} 
It has an extension as $0\leq A=A(dy, s)\in C_\ast((-\infty, 0], {\cal M}({\bf R}^2))$ with $A(dy,0)=m(x_0)\delta_0(dy)$. It becomes also a weak solution to 
\begin{equation} 
A_s=\Delta A-\nabla\cdot A\nabla\Gamma\ast A \quad \mbox{in ${\bf R}^2\times (-\infty, 0)$} 
 \label{eqn:c7}
\end{equation}  
satisfying 
\begin{equation} 
A({\bf R}^2, s)=m(x_0), \quad -\infty<s<0
 \label{eqn:c8}
\end{equation}  
with a uniformly bounded multiplicate operator. 

Now, given $\tilde s_\ell\uparrow +\infty$, we take 
\begin{equation} 
A_\ell(dy)=A(dy,-\tilde s_\ell)/m(x_0) 
 \label{eqn:52}
\end{equation} 
to apply {\it concentration compactness principle} \cite{lions84} (see also p. 39 of \cite{struwe}).  Then we obtain the following lemma, which implies Theorem \ref{thm:1.11}. 

\begin{lem}[concentration compactness]
Passing to a sub-sequence we have $m\in {\bf N}\cup\{0\}$ such that any $\varepsilon>0$ admits $y_\ell^j\in {\bf R}^2$ and $b_j>0$, $1\leq j\leq m$, satisfying 
\begin{eqnarray} 
& & \lim_{\ell\rightarrow\infty}\vert y^i_\ell-y^j_\ell\vert=+\infty, \quad \forall i\neq j \label{eqn:208-1} \\ 
& & \limsup_{\ell\rightarrow \infty}\vert A_\ell (B_\ell^j))-8\pi\vert <\varepsilon, \quad \forall j \label{eqn:208-2} \\ 
& & \vert y^j_\ell\vert \leq C_{9}(1+\max_j b_j)\tilde s_\ell^{1/2}, \quad \forall \ell\gg 1, \ \forall j \label{eqn:211}
\end{eqnarray} 
for $B_\ell^j=B(y_\ell^j, b_j)$.  Furthermore, there arises one of the following alternatives. 
\begin{enumerate}
\item $m(x_0)>8\pi m+\varepsilon$ and $A_\ell$, $\ell\gg 1$, is regular in $\displaystyle{{\bf R}^2\setminus \bigcup_{j=1}^mB(y_\ell^j, b_j)}$.  It holds that 
\begin{eqnarray} 
& & \liminf_{\ell\rightarrow\infty}A_\ell\left( {\bf R}^2\setminus \bigcup_{\ell=1}^mB_\ell^j\right)\geq m(x_0)-8\pi m-\varepsilon \label{eqn:209-2} \\ 
& & \lim_{\ell\rightarrow\infty}\Vert A_\ell\Vert_{L^\infty({\bf R}^2\setminus \bigcup_{j=1}^mB_\ell^j)}=0.  \label{eqn:209-1}
\end{eqnarray} 
\item $m(x_0)=8\pi m$ and 
\begin{equation} 
\limsup_{\ell\rightarrow\infty}A_\ell\left( {\bf R}^2\setminus \bigcup_{j=1}^mB(y_\ell^j, b_j)\right)<\varepsilon. 
 \label{eqn:210}
\end{equation} 
\end{enumerate} 
 \label{lem:3.1}
\end{lem} 

\section{Proof of Theorem \ref{thm:2.14ih}} 

From the assumption (\ref{eqn:1.209ih}), $\zeta(dy,s)$ generated in the previous section satisfies an additional condition.  Namely, each $0<\varepsilon\ll 1$ admits $s_1\gg 1$ provided with the following properties.  First, for $\tilde s\geq s_1$ there are $m(\tilde s)\in {\bf N}\cup \{0\}$, $y_j(\tilde s)\in {\bf R}^2$, and $b_j(\tilde s)>0$, $1\leq j\leq m(\tilde s)$, such that $\vert y_j(\tilde s)\vert\leq C_{10}\tilde s^{1/2}$ and $b_j(\tilde s)\leq C_{10}$. Next, $\zeta(dy,-\log \tilde s)$ is regular in $\displaystyle{ \bigcup_{\tilde s\geq \tilde s_1}({\bf R}^2\setminus E_{\tilde s})\times \{ -\log \tilde s\}}$ for 
\begin{equation} 
 E_{\tilde s}=\bigcup_{j=1}^{m(\tilde s)}B_j(\tilde s), \quad \tilde B_j(\tilde s)=B(\tilde s^{-1/2}y_j(\tilde s), \tilde s^{-1/2}b_j(\tilde s)). 
 \label{eqn:ext14}
\end{equation} 
Finally, it holds that $B_i(\tilde s)\cap B_j(\tilde s)=\emptyset$, $i\neq j$, and 
\begin{eqnarray} 
& & \sup_{\tilde s\geq s_1}\Vert \zeta(dy, -\log \tilde s)\Vert_{L^\infty({\bf R}^2\setminus E_{\tilde s})}\leq C_{11} \nonumber\\ 
& & \sup_{\tilde s\geq s_1, \ 1\leq j\leq m(\tilde s)}\left\vert \zeta(B_j(\tilde s), -\log \tilde s)-8\pi\right\vert <\varepsilon. 
 \label{eqn:2.200ih}
\end{eqnarray} 

If $m(x_0)\not\in 8\pi {\bf N}$, we have always the first alternative in Lemma \ref{lem:3.1}, which implies the existence of $\delta>0$ such that 
\begin{equation} 
\inf_{s\geq \tilde s_1}\zeta ({\bf R}^2\setminus E_{\tilde s}, -\log \tilde s)\geq \delta. 
 \label{eqn:ext13}
\end{equation} 
If (\ref{eqn:ext13}) is not the case there is $s_2\geq s_1$ such that 
\[ \zeta ({\bf R}^2\setminus E_{s_2}, -\log s_2) <\varepsilon. \] 
Then it holds that (\ref{eqn:2.189ih}) with some $\tilde t_k\uparrow T$.  Hence Theorem \ref{thm:2.14ih} is reduced to the following lemma.  For the proof we use (\ref{eqn:wsl1}), particularly, the term $\vert y\vert^2/4$, which attract the density of $\zeta(dy,s)$ to $\vert y\vert=\infty$.  

\begin{lem}
The weak scaling limit $\zeta(dy,s)\in C_\ast(-\infty, +\infty; {\cal M}({\bf R}^2))$, generated by (\ref{eqn:zeta}), does not satisfy (\ref{eqn:2.200ih}) and (\ref{eqn:ext13}), simultaneously. 
 \label{lem:2.10ih}
\end{lem}

\begin{proof} 
Similarly to Lemma \ref{pro:1} concerning (\ref{s})-(\ref{eqn:s4}), given $\tilde s_\ell\uparrow+\infty$, we have a sub-sequence denoted by the same symbol such that 
\begin{equation} 
\zeta(dy,s-\tilde s_\ell)\rightharpoonup \tilde \zeta(dy,s) \quad \mbox{in $C_\ast(-\infty, +\infty; {\cal M}({\bf R}^2))$}. 
 \label{eqn:2.204ih}
\end{equation} 
This $\tilde \zeta(dy,s)$ is a weak solution to (\ref{eqn:wsl1}). Furthermore, since $\{ \zeta(dy,s-\tilde s_\ell)\}$ is tight by (\ref{eqn:pe}), it holds that 
\begin{equation} 
\tilde \zeta({\bf R}^2,s)=m(x_0), \quad \langle \vert y\vert^2, \tilde \zeta(dy,s)\rangle \leq C_8. 
 \label{eqn:1.227-2}
\end{equation} 
We have also 
\[ \tilde s^{-1/2}\vert y_j(\tilde s)\vert\leq C_{10}, \quad \lim_{\tilde s\uparrow+\infty}\tilde s^{-1/2}b_j(\tilde s)=0 \] 
in (\ref{eqn:ext14}). 

Similarly to the remark after Lemma \ref{pro:1}, the singular part of $\tilde \zeta(dy,s)$, $s\in {\bf R}$, denoted by $\tilde \zeta_s(dy,s)$, is composed of a finite sum of delta fucntions.  By applying Lemma \ref{thm:ss02a} to the scaling back $\tilde A(dy,s)$ of $\zeta(dy,s)$, defined as in (\ref{eqn:c9}), the coefficient of each delta function of $\tilde \zeta_s(dy,s)$ must be less than or equal to $8\pi$.  These properties guarantee that the singular support of $\tilde \zeta(dy,s)$, denoted by ${\cal S}_s$, is composed of a finite number of collisionless accumulating points of $\{ \tilde s^{-1/2}y_j(\tilde s) \mid 1\leq j\leq m(\tilde s)\}$ defined for $-\log\tilde s=s+\tilde s_\ell$ as $\ell \rightarrow\infty$.  We may assume also that ${\cal S}_s$, $s\in {\bf Q}$, is composed of their converging points  by a diagonal argument. 

Therefore, we have $m(s)\in {\bf N}\cup\{ 0\}$ and $y^j_\infty(s)\in {\bf R}^2$, $1\leq j\leq m(\tilde s)$,  such that 
\[ \displaystyle{ {\cal S}_s=\{ y^j_\infty(s) \mid 1\leq j \leq m(s)\}}, \quad \tilde \zeta(\{ y^j_\infty(s)\}, s)=8\pi \] 
for $s\in {\bf R}$.  Furthermore, the density function $g=g(y,s)$ of the absolutely continuous part of $\tilde \zeta(dy,s)$ is provided with the properties 
\begin{eqnarray} 
& & 0\leq g=g(y,s) \leq C_{11}, \ \int_{{\bf R}^2}\vert y\vert^2g(y, s) \ dy \leq C_8, \ \Vert g(\cdot, s)\Vert_1\leq m(x_0) \nonumber\\ 
& & \Vert g(\cdot,s)\Vert_1\geq \delta, \ s\in {\bf Q}. 
 \label{eqn:38}
\end{eqnarray} 
It holds also that ${\cal S}_s\subset B_R$ for any $R>C_{10}$.  

Since $\tilde \zeta(dy,s)\in C_\ast(-\infty, +\infty; {\cal M}({\bf R}^2))$ the set $G=\bigcup_s({\bf R}^2\setminus {\cal S}_s)\times \{ s\}$ is open in ${\bf R}^2\times (-\infty, +\infty)$.  Furthermore, the above $g=g(y,s)$ is smooth in $G$ and it holds that 
\begin{eqnarray} 
& & g_s=\Delta g-\nabla\cdot g\nabla w \quad \mbox{in $G$} \nonumber\\ 
& & w(y,s)=\frac{\vert y\vert^2}{4}+4\sum_{j=1}^{m(s)}\log\frac{1}{\vert y-y_\infty^j(s)\vert}+\Gamma\ast g. 
 \label{eqn:1.229}
\end{eqnarray} 
Let 
\begin{equation} 
v(y,s)=\vert y\vert^2/4+(\Gamma\ast g)(y,s). 
 \label{eqn:v}
\end{equation} 
Henceforth, we put 
\begin{equation}
\overline{f}(r)=\frac{1}{2\pi}\int_0^{2\pi}f(re^{\imath \theta})d\theta 
 \label{eqn:41}
\end{equation} 
for $f=f(y)$, $y\in {\bf R}^2$, where $y=re^{\imath \theta}$ is the polar coordinate.  

We take $B_r=B(0,r)$, $r>R$, and $s_0\in {\bf R}$, to set 
\[ \displaystyle{ B^\varepsilon_r=B_r\setminus \bigcup_{j=1}^{m(s_0)}B(y_\infty^j(s_0), \varepsilon)}. \] 
By Lemma \ref{thm:ss02a}, again, any collision of the points in ${\cal S}_s$ does not occur as $s$ varies.  Therefore, making $0<\varepsilon\ll 1$, we obtain $\sharp \left(B(y^j_\infty(s_0),\varepsilon)\cap {\cal S}_s\right)\leq 1$ and 
\[ y_\infty^i(s)\in B(y^j_\infty(s_0), \varepsilon) \quad \Rightarrow \quad (y-y_\infty^i(s))\cdot \nu_y\geq 0, \ y\in \partial B(y^j_\infty(s_0), \varepsilon) \] 
for $1\leq j\leq m(s_0)$ and $\vert s-s_0\vert \ll 1$, where $\nu_y$ denotes the outer unit normal vector. 

Then, (\ref{eqn:1.229}) implies 
\begin{equation} 
\frac{d}{ds}\int_{B_r^\varepsilon}g(y,s) \ dy \leq \int_{\partial B_r^\varepsilon}g_r(y,s) \ d\sigma_y -\int_{\partial B_r^\varepsilon}(gv_r)(y,s) \ d\sigma_y
 \label{eqn:1.232-1}
\end{equation} 
for $\vert s-s_0\vert \ll 1$, where $d\sigma=d\sigma_y$ denotes the line element. We integrate (\ref{eqn:1.232-1}) in $t$, to convert it to an inequality on the difference quotient with the mesh $h>0$. Now, making $\varepsilon\downarrow 0$ and then $h\downarrow 0$, we obtain 
\begin{equation} 
\frac{d^+}{ds}\int_{B_r}g \ dy \leq \int_{\partial B_r}g_r \ d\sigma-\int_{\partial B_r}(gv_r) \ d\sigma =\int_{\partial B_r}g_r-\frac{r}{2}g -g(\Gamma\ast g)_r \ d\sigma, 
 \label{eqn:1.239+2}
\end{equation} 
where 
\[ \frac{d^+}{ds}A(s)=\limsup_{h\downarrow 0}\frac{1}{h}(A(s+h)-A(s)). \] 
Then, inequality (\ref{eqn:1.239+2}) is valid to any $r>R$ and $-\infty<s<+\infty$. 

From $\Vert g(\cdot,s)\Vert_1+\Vert g(\cdot,s)\Vert_\infty\leq C_{12}$, it follows that 
\[ \sup_s\Vert (\nabla \Gamma\ast g)(\cdot,s)\Vert_\infty<+\infty. \] 
Then, (\ref{eqn:1.239+2}) implies 
\begin{eqnarray} 
& & \frac{d^+}{ds}\int_{B_r}g \leq \int_{\partial B_r}g_r-\frac{r}{2}g+C_{13}g \ d\sigma \nonumber\\ 
& & \quad = \frac{d}{dr}(r\int_0^{2\pi}g \ d\theta)-\int_0^{2\pi}(1+\frac{r^2}{2})g \ d\theta+C_{13}r\int_0^{2\pi}g \ d\theta,  
 \label{eqn:1.240+2}
\end{eqnarray} 
which means 
\begin{equation} 
\frac{d^+}{ds}\int_0^rr\overline{g}dr\leq \frac{d}{dr}(r\overline{g})-(\frac{r^2}{2}-C_{13}r+1)\overline{g}, \quad r>R, \ -\infty<s<+\infty, 
 \label{eqn:1.241+2}
\end{equation} 
recalling (\ref{eqn:41}). Here we have 
\begin{equation} 
B(r,s)\equiv \int_0^rr\overline{g} \ dr\geq \delta-\frac{C_{14}}{r^2+1}, \quad r>R, \ s\in {\bf Q} 
 \label{eqn:1.243-3}
\end{equation}  
by (\ref{eqn:38}).  The strong maximum principle, on the other hand, implies also $B(r,s)>0$ for any $(r,s)$.  

Inequality (\ref{eqn:1.241+2}) means 
\[ \partial_s^+B\leq B_{rr}-a(r)B_r, \quad a(r)=\frac{r}{2}-C_{13}+\frac{1}{r}. \] 
Using $\displaystyle{ A(r)=\frac{r^2}{4}-C_{13}r+\log r}$, we have $a=A'$ and hence 
\[ \partial_s^+(e^{-A}B)\leq (e^{-A}B_r)_r, \quad r>R, \ -\infty<s<+\infty \] 
by $B_{rr}-a(r)B_r=e^A(e^{-A}B_r)_r$.  Now we take $r_1>R$ such that 
\[ a(r)\geq \frac{r}{4}+1, \quad \delta-\frac{C_{14}}{r^2+1}\geq \delta/2, \qquad r\geq r_1. \] 

Let $0\leq \varphi=\varphi(r)$ be a $C^1$ function on $[r_1,\infty)$, piecewise $C^2$, satisfying $\varphi(r_1)=0$, $\varphi(r)>0$ for $r>r_1$, and 
\begin{equation} 
\int_{r_1}^\infty e^{-A}\varphi \ dr<+\infty, \quad 
\lim_{r\uparrow+\infty}e^{-A}(\varphi+\varphi_r)=0. 
 \label{eqn:1.243+3}
\end{equation} 
Then we put $\varphi(r)=0$ for $r\in [0,r_1]$. Since $B\geq 0$, $B_r\geq 0$, $B(\infty,t)\leq C_{15}$, each $s\in {\bf R}$ admits $r_j\uparrow+\infty$ such that $B_r(r_j,s)\rightarrow 0$, which guarantees 
\begin{eqnarray} 
& & \frac{d^+}{ds}\int_0^\infty e^{-A}B\varphi \ dr\leq \int_0^\infty(e^{-A}B_r)_r\varphi \ dr=-\int_0^\infty e^{-A}B_r\varphi_r \ dr \nonumber\\ 
& & \quad \leq \int_0^\infty(e^{-A}\varphi_r)_rB \ dr =\int_0^\infty (\varphi_{rr}-a\varphi_r)e^{-A}B \ dr.  
 \label{eqn:1.243+2}
\end{eqnarray} 
We impose, furthermore, that the existence of $\mu>0$ such that 
\begin{equation} 
\varphi_{rr}-a\varphi_r\leq -\mu\varphi, \quad r\geq r_1 
 \label{eqn:1.243-2} 
\end{equation}
To assure all the above requirements to $\varphi(r)$, we take $0<c_1\ll 1$, for example, and put 
\[ \varphi(r)=\left\{ \begin{array}{ll} 
c_1r+c_2, & r\geq r_2 \\ 
\sin\beta(r-r_1), & r_1\leq r<r_2\equiv r_1+\frac{\pi}{4\beta} \end{array} \right. \] 
where 
\[ r_2=r_1+\frac{\pi}{4\sqrt{2}c_1}, \quad c_2=\frac{1}{\sqrt{2}}(1-\frac{\pi}{4})-r_1c_1, \quad \beta=\sqrt{2}c_1. \] 
Then we see $0\leq \varphi=\varphi(r)\in C^1[r_1,\infty)$, $\varphi(r_1)=0$, and (\ref{eqn:1.243+3}). Making $c_1\downarrow 0$, on the other hand, we obtain $r_2\uparrow+\infty$.  Therefore, (\ref{eqn:1.243-2}) arises for $0<\mu\ll 1$ by 
\[ \varphi_{rr}=\left\{ \begin{array}{ll} 
0, & r\geq r_2 \\ 
-\beta^2\varphi, & r_1\leq r<r_2 \end{array} \right. \] 
and 
\[ a(r)=\frac{r}{2}-C_{13}+\frac{1}{r}, \quad \varphi_r=c_1>0,  \qquad r\geq r_2. \] 

Since (\ref{eqn:1.243+2}) is obtained we have 
\[ \frac{d^+}{ds}\int_0^\infty e^{-A}B\varphi \ dr \leq -\mu \int_0^\infty e^{-A}B\varphi \ dr, \quad s\in {\bf R}. \] 
This means 
\[ \frac{d^+}{ds}\left\{ e^{\frac{\mu}{2} s}\int_0^\infty e^{-A}B\varphi \ dr\right\} < 0, \quad s\in {\bf R} \] 
and hence  
\[ \lim_{s\uparrow+\infty}\int_0^\infty e^{-A}B(\cdot,s)\varphi \ dr=0. \] 
We have, on the other hand, 
\[ \int_0^\infty e^{-A}B(\cdot,s)\varphi \ dr \geq \frac{\delta}{2}\int_0^\infty e^{-A}\varphi \ dr >0, \quad s\in {\bf Q} \] 
by (\ref{eqn:1.243-3}), a contradiction.  
\end{proof}

\begin{flushright}
Takashi Suzuki \\ 
Division of Mathematical Science \\ 
Department of Systems Innovation \\ 
Graduate School of Engineering Science \\ 
Osaka University \\ 
Machikaneyamacho 1-3 \\ 
Toyonakashi, 560-8531, Japan \\ 
suzuki@sigmath.es.osaka-u.ac.jp 
\end{flushright} 

\end{document}